%% file: MainFileQuasiProbabilistic.tex
\documentclass[11pt]{amsart}
\usepackage{txfonts}
\usepackage{amscd,amssymb,amsmath,latexsym,verbatim} 
\usepackage[all]{xy}
\usepackage{breakurl}
\usepackage{multirow}
\usepackage{enumitem}
\usepackage{bbm}
\usepackage{nicefrac}

\usepackage{import}
\usepackage{tikz}
\usepackage{caption}
\usepackage{subcaption}

\usepackage{color}
\usepackage[
            breaklinks=true, 
            colorlinks=true,
            linkcolor=red,
            urlcolor=blue,
            citecolor=dkgreen]{hyperref}

\usepackage{fancyhdr} 
\usepackage{titlesec}

\titleformat{\section}[block]{\normalfont\large\bfseries}
{\thesection.}{0.8em}{}

\titleformat{\subsection}[runin]{\normalfont\bfseries}
{\thesubsection.}{0.5em}{}[.]

\titleformat{\subsubsection}[runin]{\normalfont\normalsize\itshape}
{\thesubsubsection.}{0.5em}{}[.]

\usetikzlibrary{svg.path}
\definecolor{gG}{RGB}{ 60, 186,  84}
\definecolor{gY}{RGB}{244, 194,  13}
\definecolor{gB}{RGB}{ 72, 133, 237}
\definecolor{gR}{RGB}{219,  50,  54}

\newlength\htG
\protected\def\google{\settoheight{\htG}{G}%
  \begin{tikzpicture}[yscale=-1,scale=(\htG/240pt),baseline=(baseline)]
    \fill[fill=gG] svg {m797.49 249.7h35.975v-240.75h-35.975z};
    \coordinate (baseline) at (current bounding box.south);
    \fill[fill=gB] svg {m246.11 116.18h-116.57v34.591h82.673c-4.0842 48.506-44.44 69.192-82.533 69.192-48.736 0-91.264-38.346-91.264-92.092 0-52.357 40.54-92.679 91.371-92.679 39.217 0 62.326 25 62.326 25l24.22-25.081s-31.087-34.608-87.784-34.608c-72.197-0.001-128.05 60.933-128.05 126.75 0 64.493 52.539 127.38 129.89 127.38 68.031 0 117.83-46.604 117.83-115.52 0-14.539-2.1109-22.942-2.1109-22.942z};
    \fill[fill=gR] svg {m341.6 91.129c-47.832 0-82.111 37.395-82.111 81.008 0 44.258 33.249 82.348 82.673 82.348 44.742 0 81.397-34.197 81.397-81.397 0-54.098-42.638-81.959-81.959-81.959zm0.47563 32.083c23.522 0 45.812 19.017 45.812 49.66 0 29.993-22.195 49.552-45.92 49.552-26.068 0-46.633-20.878-46.633-49.79 0-28.292 20.31-49.422 46.741-49.422z};
    \fill[fill=gY] svg {m520.18 91.129c-47.832 0-82.111 37.395-82.111 81.008 0 44.258 33.249 82.348 82.673 82.348 44.742 0 81.397-34.197 81.397-81.397 0-54.098-42.638-81.959-81.959-81.959zm0.47562 32.083c23.522 0 45.812 19.017 45.812 49.66 0 29.993-22.195 49.552-45.92 49.552-26.068 0-46.633-20.878-46.633-49.79 0-28.292 20.31-49.422 46.741-49.422z};
    \fill[fill=gB] svg {m695.34 91.215c-43.904 0-78.414 38.453-78.414 81.613 0 49.163 40.009 81.765 77.657 81.765 23.279 0 35.657-9.2405 44.796-19.847v16.106c0 28.18-17.11 45.055-42.936 45.055-24.949 0-37.463-18.551-41.812-29.078l-31.391 13.123c11.136 23.547 33.554 48.103 73.463 48.103 43.652 0 76.922-27.495 76.922-85.159v-146.77h-34.245v13.836c-10.53-11.347-24.93-18.745-44.04-18.745zm3.178 32.018c21.525 0 43.628 18.38 43.628 49.768 0 31.904-22.056 49.487-44.104 49.487-23.406 0-45.185-19.005-45.185-49.184 0-31.358 22.619-50.071 45.66-50.071z};
    \fill[fill=gR] svg {m925.89 91.02c-41.414 0-76.187 32.95-76.187 81.57 0 51.447 38.759 81.959 80.165 81.959 34.558 0 55.768-18.906 68.426-35.845l-28.235-18.787c-7.3268 11.371-19.576 22.484-40.018 22.484-22.962 0-33.52-12.574-40.061-24.754l109.52-45.444-5.6859-13.318c-10.58-26.08-35.26-47.86-67.92-47.86zm1.4268 31.413c14.923 0 25.663 7.9342 30.224 17.447l-73.139 30.57c-3.1532-23.667 19.269-48.017 42.915-48.017z};
  \end{tikzpicture}%
}

\newcounter{first}
{\end{list}}

\definecolor{dkgreen}{rgb}{0.1,0.4,0.0}
\definecolor{dkblue}{rgb}{0,0.1,0.8}
\definecolor{dkred}{rgb}{1,0,0}

\def\define{\ensuremath{\overset{\operatorname{\scriptscriptstyle def}}=}}

\newenvironment{alphenum}%
{ 
 
 \begin{enumerate}}{\end{enumerate}}

\newenvironment{romenum}
{

\begin{enumerate}}{\end{enumerate}}

\theoremstyle{plain}
\newtheorem{theorem}{Theorem}[section]

\theoremstyle{definition}

\newtheorem{definition}{Definition}[section]

\newtheorem*{exampleintro*}{Example}

\newtheorem*{conjecture*}{Conjecture}

\newcommand{\doi}[1]
{\texttt{\href{http://dx.doi.org/#1}{\nolinkurl{doi:#1}}}}
\newcommand{\web}[1]
{\texttt{\href{#1}{\nolinkurl{#1}}}}

\title{Cooperative games on simplicial complexes}
\author{Ivan Martino}
\date{\today}

\begin{document}

\begin{abstract}
In this work, we define cooperative games on simplicial complexes, generalizing the study of probabilistic values of Weber \cite{Weber-robabilistic-values-for-games} and quasi-probabilistic values of Bilbao, Driessen, Jim\'{e}nez Losada and Lebr\'{o}n \cite{Shapley-matroids-static}.

\noindent
Applications to Multi-Touch Attribution and the interpretability of the Machine-Learning prediction models motivate these new developments \cite{NIPS2017_7062, Ribeiro, Erik-Igor, Shrikumar, 7546525, OnPixel}.

We deal with the axiomatization provided by the $\lambda_i$-dummy and the monotonicity requirements together with a probabilistic form of the symmetric and the efficiency axioms.

\noindent
We also characterize combinatorially the set of probabilistic participation influences as the facet polytope of the simplicial complex.
\end{abstract}

\maketitle

\input{./definitions.tex}


A cooperative game is a pair $(n, v)$ where $n$ is a positive integer and $v$ is the \emph{worth} function $v: 2^n \rightarrow \mathbb{R}$, where $2^n$ is the power set of $[n]\define \{1, \dots, n\}$. We assume that $v(\emptyset)=0$.
The elements of $[n]$ are players of the game that may join in coalitions;
a coalition $T$ is a subset of $[n]$ and $v(T)$ is the number of payoff of $T$ in the cooperative game. 

\noindent
For every player $i$ an individual value $\phi_i(v)$ is a (linear) function measuring the additional \emph{worth} that $i$ provides to a coalition during the cooperative game $(n, v)$.
%
The study of such values was extremely relevant for the community and we would like to highlight here a few important works of Shapley \cite{Shapley-a-value, Shapley-core-convex} and Weber  \cite{Weber-robabilistic-values-for-games} that have influenced the author.

Recently, quite a lot of effort has been done to study cooperative games on matroids \cite{Shapley-matroids-static, Shapley-matroids-dynamic, MR3886659, MR2847360, MR2825616, MR1436577, MR1707975}.
%
Inspired by this recent articles, in this manuscript we define cooperative games on simplicial complexes and we study quasi-probabilistic values for such games.

\vspace{0.2cm}

A simplicial complex is a family $\Delta$ of subsets of $[n]$ such that if $X\in \Delta$, then every subset $Y\subseteq X$ will also belong to $\Delta$. For instance, every graph and the full power set $2^{n}$ are simplicial complex. In the latter case, $\Delta=2^{n}$ is called a $(n-1)$-dimensional simplex.
The reader may find more examples all along the paper, but also highlighted in Figures \ref{fig:figure-example-2}, \ref{fig:figure-matroid-yes}, and \ref{fig:figure-matroid-no}.

\noindent
A cooperative game on $\Delta$ is defined by a worth function $v$:
\[
	v: \Delta \rightarrow \mathbb{R}
\]
with the usual constrain that $v(\emptyset)=0$.
The traditional game $(n, v)$ can be seen as the cooperative game on the $(n-1)$-dimensional simplex $(2^n, v)$, where the function $v$ is the defines as in the classical case.

\noindent
In other words, in a cooperative game on a a simplicial complex $\Delta$ a player $i$ in $[n]$ may join a coalition $T$ only if $T\cup i\in \Delta$. In such case, the coalition is \emph{feasible} and, \emph{unfeasible} otherwise.
Similarly as in the classical case, the individual function $\phi_i(v)$ measures the additional value that $i$ provide to a feasible coalition during the cooperative game $(\Delta, v)$.

\section*{Presentation of the results}
We are going to present the new research developments and after this we provide several motivation for this work.
Further, we explain carefully a concrete application to Multi-Touch Attribution and for this prototype example we show explicitely all the new introduced objects. 

\vspace{0.1cm}
A player $i$ is \emph{dummy} for the cooperative game $(\Delta, v)$ if the player does not provide better results to the coalition then its own worth $v(\{i\})$. In mathematical terms, $i$ is a dummy player in the cooperative game $(\Delta, v)$ if $v(S\cup \{i\})=v(S)+v(\{i\})$ for every $i\notin S$ and  $S\cup \{i\}\in \Delta$.
For simplicity we are going to neglect the set brackets for the singleton and write $S\cup i$ instead of $S\cup \{i\}$ and $v(i)$ instead of $v(\{i\})$.

\noindent
Moreover, we say that the worth function $v$ is \emph{monotone} if $S \subseteq T \in \Delta$ implies $v(S)\leq v(T)$.
We denote by $\charFun$ the set of all cooperative game on the simplicial complex $\Delta$.

Theorem 2 and 3 of \cite{Weber-robabilistic-values-for-games} characterize the individual values that satisfies the dummy and the monotonicity axioms. 
Theorem 3.1 of \cite{Shapley-matroids-static} generalizes such results for cooperative game on a matroids. Every matroid is a simplicial complex but not vice versa, see for instance the example in Figure \ref{fig:figure-example-2}. 
Next theorem moves further and extends the results to every simplicial complex $\Delta$.

We denote by $0\leq \lambda_i \leq 1$ the rate of participation of the player $i$ in the cooperative game $(\Delta, v)$ and let us rewrite the two main conditions in Theorem 3.1 of \cite{Shapley-matroids-static}: 
\begin{description}
	\item [\hspace{1.50cm} $\lambda_i$-Dummy axiom] If the player $i$ is dummy for $(\Delta, v)$, then $\phi_i(v)=\lambda_i v(i)$;
	\item [\hspace{1cm} Monotonicity axiom] If $v$ is monotone, then $\phi_i(v)\geq 0$.
\end{description}

\setcounter{section}{2}
\setcounter{theorem}{0}
\begin{theorem}
	Let $\Delta$ be a simplicial complex on $n$ verticies  and let $\phi_i$ be an individual value for a player $i$.
	The individual value $\phi_i$ is a $\mathbb{R}$-linear function satisfying the $\lambda_i$-Dummy axiom and the Monotonicity axiom if and only if there exists a collection of non-negative real numbers $\{p_T^i: T\in \Link{i}{\Delta}\}$ with 
		\begin{equation*}
			\sum_{T\in \Link{i}{\Delta}} p_T^i = \lambda_i	
		\end{equation*}	
		such that for all $v$ in $\charFun$,
		\begin{equation*}
			\phi_i(v)=\sum_{T\in \Link{i}{\Delta}} p_T^i (v(T \cup i) - v(T)).		
		\end{equation*}
\end{theorem}

The necessary and sufficient condition of the previous theorem hints a natural generalization of \emph{probabilistic values} for cooperative games on simplicial complexes. For this reason,  
Bilbao, Driessen, Jim\'{e}nez Losada and Lebr\'{o}n \cite{Shapley-matroids-static} called, in the case of matroids, those values \emph{qusi-probabilistic}, see Section \ref{sec:quasi-probabilistic}.

Recall that $\FacetsD$ is the set of facets of the simplicial complex $\Delta$, these are sets in $\Delta$ that are maximal by inclusion.
Moreover, $\Facet{i}{\Delta}$ collects all the facets containing $i$.
In addition, we denote by $\bar{F}$ the simplex on the verticies in $F$, i.e. $\bar{F}=2^F$.

\setcounter{section}{2}
\setcounter{theorem}{1}
\begin{theorem}
	Let $\Delta$ be a simplicial complex and let $\phi_i$ be an individual value for a player $i$ in $[n]$.
	%

	The individual value $\phi_i$ is a quasi-probabilistic value if and only if there exists a probability distribution $\{P^i(F_1), \dots, P^i(F_k)\}$ on $\FacetsD$ such that
	\begin{equation*}
		\sum_{F\in \Facet{i}{\Delta}} P^i(F) = \lambda_i,	
	\end{equation*}	
	and for every $F\in \Facet{i}{\Delta}$ there exists an individual (classical) probabilistic value $\phi_i^F$ defined on the simplex $\bar{F}$ such that for all $v$ in $\charFun$,
	\begin{equation*}
		\phi_i(v)=\sum_{F\in \Facet{i}{\Delta}} P^i(F) \phi_i^F(v|_F),		
	\end{equation*}
	where $v|_F$ is the restriction of the cooperative game $(\Delta, v)$ to $(F, v|_F)$.
\end{theorem}
\noindent
The previous statement extends Theorem 3.2 in \cite{Shapley-matroids-static} to every simplicial complex.

As in the traditional case, we collect all individual values together, in the \emph{group value} $\phi=(\phi_1, \phi_2, \dots, \phi_n)$.
For sake of presentation of the results, let us assume here that we work in the \emph{efficient} scenario, that is the individual values are nor optimistic or pessimistic.
%
A \emph{group value} $\phi$ for the cooperative game $(\Delta, v)$ is \emph{reducible} if there exists a probability distribution $P$ on the the facets of $\Delta$ such that 
\begin{equation*}
	\phi_i(v)=\sum_{F\in \Facet{i}{\Delta}} P(F) \phi_i^F(v|_F).		
\end{equation*}
where $\phi_i^F$ is a probabilistic value for a cooperative game on the simplex $\bar{F}$ and $v|_F$ is the restriction of the characteristic function $v$ on the simplex $F$, that is $v|_F(S)=v(S)$ for every subset $S$ of $F$.

\noindent
This notion was introduced in \cite{Shapley-matroids-static} as \emph{basic value}. Our choice for the adjective \emph{reducible} is made in view of the results in 
\cite{Martino-Efficiency, Martino-Probabilistic-value}. 

We denote by $\operatorname{Prob}{\Delta}$ the set of probability distribution over the set of facets of $\Delta$. 
Following Section 4 of\cite{Shapley-matroids-static}, we also define the \emph{probabilistic participation influence} $w^P(T)$ for the coalition $T$ as 
\[
	w^P(T)\define \sum_{T\subset F\in \FacetsD} P(F)
\]

\noindent
Using the result of Edmonds \cite{Edmonds-submodular-functions} in \eqref{eq:anti-cor-poly}, Bilbao, Driessen, Jim\'{e}nez Losada and Lebr\'{o}n \cite{Shapley-matroids-static} are able to show 
that 
\[
	\operatorname{anticore} ([n], \operatorname{rk}_{M}) \overset{\operatorname{\scriptscriptstyle \eqref{eq:anti-cor-poly}}}=Q_{M} =\{w^P: P\in \operatorname{Prob}{M}\},
\]
where $Q_{M}$ is \emph{facet polytope} of the matroid $M$, see Section \ref{sec:core}. 

Let $Q_{\Delta}$ be the convex hull in $\mathbb{R}^n$ of vectors $e_F=\sum_{i\in F}e_i$ for every facet F in \FacetsD, where $e_i$ is a standard orthonormal basis of $\mathbb{R}^n$, see Definition \ref{def:face-polytope}.

While the equality on the left does not hold for simplicial complexes, we are able to prove that $Q_{\Delta}=\{w^P: P\in \operatorname{Prob}{\Delta}\}$.
Aside of the generalization per se of Proposition 4.1 of \cite{Shapley-matroids-static}, one perk of the proof of next proposition is that we do not use Edmond result or any connection with the anticore of the cooperative game. 

\setcounter{section}{4}
\setcounter{theorem}{1}
\begin{theorem}
	Let $\Delta$ be a simplicial complex and let $r$ be its rank function. Then
	\[
		Q_{\Delta}=\{w^P: P\in \operatorname{Prob}{\Delta}\}.
	\]
\end{theorem}

All these results together allow us to generalize the main theorem of \cite{Shapley-matroids-static} to pure simplicial complex, i.e. simplicial complex where all the facets have the same cardinality. The next statement characterizes when an individual value is reducible to classical cooperative games defined on the facets of the pure simplicial complex. 
To do this, we need two additional constrains:
 \begin{description}
	\item [\hspace{0.7cm} Substitution for carrier games] For every feasible coalition $T$ and for every pair of players, one has $\phi_i(v_T)=\phi_j(v_T)$ where $v_T$ is the carrier game defined in Definition \ref{def:carrier-games};
	\item [\hspace{0.7cm} Probabilistic efficiency] For every cooperative game $(\Delta, v)$, $\sum_i \phi_i(v)=\sum_{F\in \FacetsD}P(F)v(F)$.
\end{description}

\setcounter{section}{5}
\setcounter{theorem}{0}
\begin{theorem}
	Let $\Delta$ be a pure simplicial complex on $n$ verticies and and let $\phi_i$ be the individual value for a player $i$.
	The group value $\phi$ is reducible and it decomposes as the weighted sum of Shapley values,
	\begin{equation*}
		\phi_i(v)=\sum_{F\in \Facet{i}{\Delta}} P(F)\operatorname{Shapley}^F(v|_F).	
	\end{equation*}
	if and only if each $\phi_i$ is a linear function that satisfies the $w^P(i)$-dummy axiom and the group value fulfills the \emph{Substitution for carrier games} and the \emph{Probabilistic efficiency} axioms. 
\end{theorem} 

The \emph{substitution} axiom is the one enforcing that $\Delta$ is a pure simplicial complex; the probabilistic efficiency provides, instead, the weights of the decomposition of $\phi_i$ as sum of Shapley values.

\section*{Motivations}
\noindent
It is worth to mention a few reasons why the generalization provided is relevant.

\textbf{i)} Not every pure simplicial complex is a matroid. Figure \ref{fig:figure-matroid-yes} and \ref{fig:figure-matroid-no} show two pure simplicial complex and only the one in the left is a matroid.
Thus, the new results extend quite substantially the results in Theorem 3.1, Theorem 3.2 and Theorem 4.2 of \cite{Shapley-matroids-static}.

\begin{figure}
\centering
\begin{subfigure}[b]{5cm}
\begin{tikzpicture}[scale=0.70]
   	\coordinate (F) at ( -3.0,  1.0);
  	\coordinate (S) at (-3.0, -2.0); 
	\coordinate (FB) at ( 0.0, -0.0); 
	\coordinate (TV) at (3.0, 1.0); 
	\coordinate (E) at (3.0, -2.0);
	 
   \node [left] at ( -3.0,  1.0) {1}; 
   \node [left] at (-3.0, -2.0) {2}; 
	\node [above] at ( 0.0, -0.0) {3}; 
   \node [right] at (3.0, 1.0) {4}; 
   \node [right] at (3.0, -2.0) {5};

   \filldraw [draw=black, fill=red!20, line width=1.5pt] (FB)--(TV)--(E) -- (FB);
   
   \filldraw [draw=black, fill=green!20, line width=1.5pt] (FB)--(S)--(E) -- (FB);

\filldraw [draw=black, fill=blue!20, line width=1.5pt] (FB)--(F)--(S) -- (FB); 
   
   
\end{tikzpicture}
        \subcaption{This simplicial complex is a matroid.}
        \label{fig:figure-matroid-yes}
    \end{subfigure}
~\hspace{0.5cm}
	\begin{subfigure}[b]{5cm}
\begin{tikzpicture}[scale=0.75]
   	\coordinate (F) at ( -3.0,  1.0);
  	\coordinate (S) at (-3.0, -2.0); 
	\coordinate (FB) at ( 0.0, -0.0); 
	\coordinate (TV) at (3.0, 1.0); 
	\coordinate (E) at (3.0, -2.0);
	 
   \node [left] at ( -3.0,  1.0) {1}; 
   \node [left] at (-3.0, -2.0) {2}; 
	\node [above] at ( 0.0, -0.0) {3}; 
   \node [right] at (3.0, 1.0) {4}; 
   \node [right] at (3.0, -2.0) {5};

   \filldraw [draw=black, fill=red!20, line width=1.5pt] (FB)--(TV)--(E) -- (FB);

   \filldraw [draw=black, fill=blue!20, line width=1.5pt] (FB)--(S)--(F) -- (FB);
   
\end{tikzpicture}
        \subcaption{This simplicial complex is not a matroid.}
        \label{fig:figure-matroid-no}
	\end{subfigure}

\caption{Two examples of pure simplicial complexes.}
\end{figure}
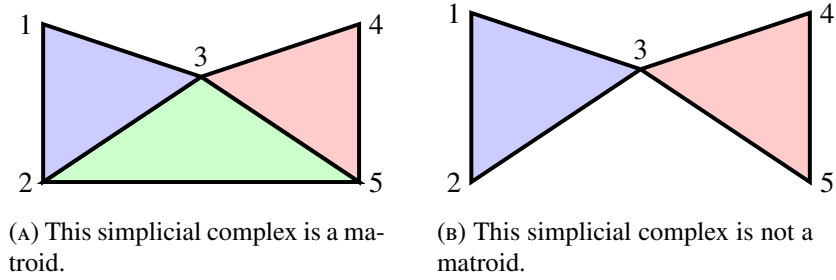

\vspace{0.1cm}
\textbf{ii)} Aside the generalization from matroids to simplicial complexes in it-self, our proofs of the results show that the main matroidal properties are not useful in this context \cite{Stanley2012b, Stanley1996a,  MR782306, Oxley} and underline the importance of the link of the player, see Definition \ref{def:link}. This will be crucial in the subsequent works 
\cite{Martino-Efficiency,Martino-Probabilistic-value}.

\vspace{0.1cm}
\textbf{iii)} A ground-breaking application of the Shapley value methodology is in the interpretability of the Machine-Learning prediction models, see for instance \cite{NIPS2017_7062} and the Phyton package SHAP \href{https://github.com/slundberg/shap/blob/master/README.md#citations}{[Github repository]}, see  \cite{Ribeiro, Erik-Igor, Shrikumar, 7546525, OnPixel, Lipovetsky}.

\vspace{0.1cm}
\textbf{iv)} Last but not the least, another very important application is that \textsc{Google} has started to use it in its own multi-touch attribution system offered in \textit{Google 360}. (This is the marketing platform developed and offered by \google.)

\section*{The prototype example}
We provide here a prototype motivating application of the new results: the \emph{Multi-Touch Attribution} in Marketing. 
In Marketing attribution the individual values are used to understand what \emph{set} of advertisements have influenced a person toward  a desired behavior, typically becoming a costumer.
This is (classically) described as the following cooperative game.
$\Delta=2^n$ is the simplex with verticies in the set of marketing channels, labeled for simplicity from $1$ to $n$. 
To each subset $T$ of $[n]$, the worth function $v$ associates the number of conversions $v(T)$, people who became costumers, influenced by the marketing channels in $T$. In other words, $v(T)$ counts the number of individuals who decided to buy being effected by the advertisements collected in the set $T$. 

\noindent
Unfortunately, the marketing team cannot always track down $v(T)$ for every subset of channels, and setting $v(T)=0$ for those untraceable \emph{coalitions} may results in technical problems such as $v$ is not anymore monotone, that is $v(S)\leq v(T)$ for every $S\subseteq T\subseteq [n]$.
The idea we propose is to mark such coalition as unfeasible and to take them out from $\Delta$. We still assume that if $T$ is a feasible coalition, then every subset of $T$ is still feasible. Thus, $\Delta$ is a simplicial complex, see Section \ref{sec:simplicial-complex}.

\begin{exampleintro*}
		The advertising campaign of a certain store is made of six different marketing channels: the distribution of flyers (F) and an advertising stand (S) in the weekly market of the district, together with social network (FB), email (E), TV and search engine (G) advertisements.
	The shop also offers a discount for on-line purchases that can be retrieved with a promotion code in the flyers or by request in the stand. 
	%
%
	
	\noindent
	Because of the nature of the data, the marketing team cannot analyze if conversions (individuals who make a purchase decision) have been exposed by all the channels and in particular, they can only track the following ones: every subsets of the channels $\{FB, S, F\}$ and every subsets of $\{FB,TV,E, G\}$. 
	Figure \ref{fig:figure-example-2} shows the simplicial complex $\Delta$ of feasible coalitions:
	\[
	\begin{split}
	\Delta=&\left\{\emptyset, \{F\}, \{S\}, \{FB\}, \{E\}, \{TV\}, \{S, F\}, \{FB, F\}, \{FB, S\},\{FB, S, F\}, \right. \\
	&\{FB, E\}, \{FB, TV\}, \{TV, E\}, \{FB, TV, E\},\{FB, G\}, \{TV, G\}, \\
	&\left.  \{FB, TV, G\},	\{E, G\}, \{FB, E, G\}, \{TV, E, G\} \right\}.
	\end{split}	
	\]
\end{exampleintro*} 

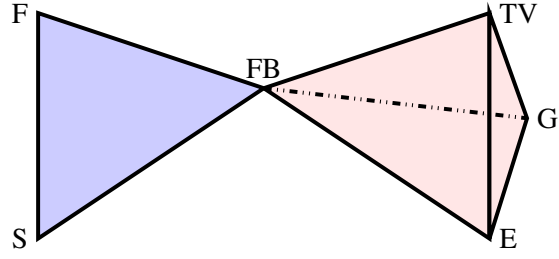
\begin{figure}
\centering
\begin{tikzpicture}[scale=1]
   	\coordinate (F) at ( -3.0,  1.0);
  	\coordinate (S) at (-3.0, -2.0); 
	\coordinate (FB) at ( 0.0, -0.0); 
	\coordinate (TV) at (3.0, 1.0); 
	\coordinate (E) at (3.0, -2.0);
	\coordinate (G) at (3.5, -0.4);

   \node [left] at ( -3.0,  1.0) {F}; 
   \node [left] at (-3.0, -2.0) {S}; 
	\node [above] at ( 0.0, -0.0) {FB}; 
   \node [right] at (3.0, 1.0) {TV}; 
   \node [right] at (3.0, -2.0) {E}; 
   \node [right] at (3.5, -0.4) {G};



   \filldraw [draw=black, fill=red!10, line width=1.5pt] (FB)--(TV)--(E) -- (FB);
	
   \filldraw [draw=black, fill=red!10, line width=1.5pt] (E)--(TV)--(G) -- (E);

	\draw[draw=black,dash dot dot, line width=1.5pt] (FB) -- (G);   
   
   \filldraw [draw=black, fill=blue!20, line width=1.5pt] (FB)--(S)--(F) -- (FB);
   
\end{tikzpicture}
\caption{The simplicial complex describing the feasible coalitions for the cooperative game of the prototype example.}
\label{fig:figure-example-2}
\end{figure}

\section*{Breaking the prototype example down}

Before moving to the preliminary definitions and proving all the presented results, we want to provide the reader a concrete example by computing the quasi-probabilistic values for the cooperative games in the previous example.

Assume that the cooperative game satisfies the conditions in Theorem \ref{thm:dummy_and_monotone} and enumerate the set of marketing channels as $\{F, S, FB, TV, E, G\}=\{1, 2,3,4,5, 6\}$. 


If we suppress the brackets notation for set and write, for instance $345$ for the set $\{3,4,5\}$. 
Then we write down the individual values as:
\begin{eqnarray*}
	\phi_1(v)&=& p_\emptyset^1 v(1)+p_2^1 (v(12) - v(2))	+ p_3^1(v(13) - v(3)) + p_{23}^1(v(123) - v(23));\\
	\phi_2(v)&=& p_\emptyset^2 v(2)+p_1^2 (v(12) - v(1))	+ p_3^2(v(23) - v(3)) + p_{13}^2(v(123) - v(13));\\
	\phi_3(v)&=& p_\emptyset^3 v(3)+p_1^3 (v(13) - v(1)) + p_2^3 (v(23) - v(2)) + p_{12}^3(v(123) - v(12)+\\
	&+& p_4^3(v(43) - v(4))+ p_5^3(v(35) - v(5))+ p_6^3(v(36) - v(6)) +\\ 
	&+& p_{45}^3(v(345) - v(45))+ p_{46}^3(v(346) - v(46))+p_{56}^3(v(356) - v(56));\\
	\phi_4(v)&=& p_\emptyset^4 v(4)+p_3^4 (v(34) - v(3)) + p_5^4(v(45) - v(5))+ p_6^4(v(46) - v(6))+ \\
	&+& p_{35}^4(v(345) - v(35))+p_{36}^4(v(346) - v(36))+p_{56}^4(v(356) - v(56));\\
	\phi_5(v)&=& p_\emptyset^5 v(5)+p_3^5 (v(35) - v(3)) + p_4^5(v(45) - v(4))+ p_6^5(v(56) - v(6))+ \\
	&+& p_{34}^5(v(345) - v(34))+p_{36}^5(v(356) - v(36))+p_{46}^5(v(456) - v(46));\\
	\phi_6(v)&=& p_\emptyset^6 v(6)+p_3^6 (v(36) - v(3)) + p_5^6(v(56) - v(5))+ p_4^6(v(46) - v(6))+ \\
	&+& p_{35}^6(v(356) - v(35))+p_{34}^6(v(346) - v(36))+p_{45}^6(v(456) - v(45)).
\end{eqnarray*}

Since $\Delta$ is not pure, $\phi$ cannot be reduced as sum of Shapley values.
It is also worth to mention that the face polytope $Q_{\Delta}$ that describes the convex all of probabilistic participation influence is a polytope in $\mathbb{R}^6$ defined by $20 = 6+9+5+2$ points.

%
%
%
%
%
%
%
%
%
%
%

\subsection*{Acknowledgments} 
	The author is currently supported by the Knut and Alice Wallenberg Foundation and by the Royal Swedish Academy of Science.
	
\setcounter{section}{0}
\setcounter{theorem}{0}
\section{Preliminaries}
Let $n$ be a positive integer and we denote by $[n]\define\{1,\dots, n\}$.

\subsection{Simplicial Complexes}\label{sec:simplicial-complex}
A (finite) simplicial complex $\Delta$ over $n$ verticies is a family of subsets in $2^{[n]}$ such that if a set $T$ is in $\Delta$, then any of its subsets $S\subset T$ will be in $\Delta$ too. 
There is a natural rank function, $\rank_{\Delta}$, on $\Delta$ provided by the cardinality of its sets. The elements of $\Delta$ that are maximal by inclusions are called \emph{facets}. 
The maximal value of this rank function is the rank of the simplicial complex $\rank{\Delta}$.
%

All along this work $\Delta$ is a finite simplicial complex over $n$ verticies of rank $r\define \rank{\Delta}$. 
Let $\FacetsD$ be the set of facets of $\Delta$, $\FacetsD=\{F_1, \dots, F_k\}$ and for every $T\in \Delta$ let $\Facet{T}{\Delta}$ be the set of facets in $\Delta$ that contain $S$.

%
%
%
If $S$ is an element in $\Delta$, then $\bar{S}\define\{T: T\subseteq S\}=2^{S}$ is the simplicial complex obtained by all sets contained in $S$. 
%
The next definition is extremely important for this work.

%
%

\begin{definition}\label{def:link}
The link of of an element $S$ in a simplicial complex $\Delta$ is made by the subsets $A$ of $T\in \Delta$, such that $T$ is disjoint by $S$ and can be completed by $S$, $S\cup T$, to an element in $\Delta$:

\begin{equation*}
	\Link{S}{\Delta}=\{A:\, A\in \bar{T} \mbox{ with } T\in \Delta \mbox{ such that } S\cap T = \emptyset, S\cup T \in \Delta\}.
\end{equation*}
\end{definition}

\noindent
The case when $S$ is the singleton $\{i\}$ will be extremely relevant in our work: $\Link{i}{\Delta}$ is the set of simplex $T$ in $\Delta$ with $i\notin T$ such that $T\cup i\in \Delta$:
\[
	\Link{i}{\Delta}=\{T\in \Delta: i\notin T \mbox{ and } T\cup i \in \Delta\}.
\]


%

%
%
%

\subsection{Cooperative games on simplicial complexes}
A cooperative game on the simplicial complex $\Delta$ is the pair $(\Delta, v)$ where $v$ is a worth function $v:\Delta\rightarrow \mathbb{R}$ under the constrain $v(\emptyset)=0$. (Here we mean that the function is defined on $\operatorname{Set}(\Delta)$, that is the collection of elements of $\Delta$ forgetting the partial order.)
The verticies $[n]$ of the simplicial complex are the players of the cooperative game and a coalition $T$ is feasible if $T\in \Delta$.
The set $\charFun$
of characteristic functions on $\Delta$ is naturally a real vector space.

Two types of games have a very important role for the theory of probabilistic values, see \cite{Weber-robabilistic-values-for-games}. By abuse of notation we call both families \emph{carrier games} even if the notion usually refer only to the first family:
\[
	\mathcal{C} = \{v_T: \emptyset \neq T \subset [n] \}, \,\,\,\hat{\mathcal{C}} = \{\hat{v}_T: \emptyset \neq T \subset [n] \},
\]
where $v_T$ and $\hat{v}_T$ are so defined:
\[
	v_T (S)=\begin{cases}
				1 &  T\subseteq S\\
				0 &\text{otherwise.}
			\end{cases},\,\,\, \hat{v}_T (S)=\begin{cases}
				1 &  T\subsetneq S\\
				0 &\text{otherwise.}
			\end{cases}
\]
We generalize this definition for any element $T$ of a simplicial complex. Indeed for every partially order set $(P, \leq_P)$ and every element $q$ in $P$ we consider the following function:
\[
	u_q^P(s)\define \begin{cases}
				1 &  q\leq_P s \\
				0 &\text{otherwise.}
			\end{cases},\,\,\, 				\hat{u}_q^P(s)\define \begin{cases}
				1 &  q<_P s \\
				0 &\text{otherwise.}
			\end{cases}
\]
Thus, we define
\[
	v_T (S)\define u_T^\Delta(S),\,\,\, \hat{v}_T (S)\define \hat{u}_T^\Delta (S).
\]
It is easy to see that in the classical case (when $\Delta$ is a full simplex on $n$ verticies) these functions reproduce the carrier games.

\begin{definition}\label{def:carrier-games}
	Let $\Delta$ be a simplicial complex. The sets of carrier games are so defined:
	\[
		\mathcal{C} = \{v_T: \emptyset \neq T \in \Delta \}, \,\,\,\hat{\mathcal{C}} = \{\hat{v}_T: \emptyset \neq T\in \Delta \}, 
	\]
	where $v_T (S)\define u_T^\Delta(S)$ and $\hat{v}_T (S)\define \hat{u}_T^\Delta (S)$.
\end{definition}

\noindent
When $T$ is the empty set, we define $\hat{v}_{\emptyset}\define \hat{u}_{\emptyset}^\Delta$.

We also denote by $\mathbbm{1}_T$ the indicator function of the set $T$ in $\Delta$, that is:
\[
	\mathbbm{1}_T(S)\define \begin{cases}
			1 &  T=S \in \Delta \\
			0 &\text{otherwise.}
		\end{cases}.
\]
We observe that $\mathbbm{1}_T=v_T-\hat{v}_T$.

\begin{definition}
An \emph{individual value} for a player $i$ in $[n]$ is a function $\phi_i:\charFun\rightarrow \mathbb{R}$.
\end{definition}

The first axiom in the theory of probabilistic values is that these functions are linear, see Section 3 of \cite{Weber-robabilistic-values-for-games}.
We are going to only consider linear individual values.

\input{./QPvalue.tex}

 	\vspace{0.5cm}

\bibliographystyle{amsalpha}
\bibliography{unica}

 	\vspace{0.5cm}
	
 	\noindent
 	{\scshape Ivan Martino}\\
 	{\scshape Department of Mathematics, Royal Institute of Technology.}\\ 
 	{\itshape E-mail address}: \texttt{imartino@kth.se}

\end{document}

%% file: definitions.tex
\def\rank#1{\ensuremath{\operatorname{rank} #1}}
\def\face{\ensuremath{\operatorname{\mathcal F} (\Delta)}}

\def\FacetsD{\ensuremath{\operatorname{Facets} \Delta}}	
\def\Facets#1{\ensuremath{\operatorname{Facets} #1}}
\def\Facet#1#2{\ensuremath{\operatorname{Facets}_{#2} #1}}

\def\closure#1#2{\ensuremath{\operatorname{cl}_{#1} (#2)}}
\def\closureDelta#1{\ensuremath{\operatorname{cl}_{\Delta} (#1)}}

\def\Star#1#2{\ensuremath{\operatorname{Star}_{#2} #1 }}
\def\StarNoD#1{\ensuremath{\operatorname{Star} #1 }}
\def\StarD#1{\ensuremath{\operatorname{Star}_{\Delta} #1 }}

\def\Link#1#2{\ensuremath{\operatorname{Link}_{#2} #1 }}
\def\LinkNoD#1{\ensuremath{\operatorname{Link} #1 }}
\def\LinkD#1{\ensuremath{\operatorname{Link}_{\Delta} #1 }}

\def\ffD{\ensuremath{\operatorname{\textbf{f}}(\Delta)}}
\def\ff#1{\ensuremath{\operatorname{\textbf{f}}(#1)}}
\def\ffi#1#2{\ensuremath{\operatorname{f}_{#1}(#2)}}

\def\charFun{\ensuremath{\operatorname{\mathbb{R}} (\Delta) }}

\def\vtot{\ensuremath{\operatorname{v}_{\Delta}}}
\def\vtotof#1{\ensuremath{\operatorname{v}_{#1}}}

\def\Coalition#1#2{\ensuremath{\operatorname{Coalition}_#2 #1 }}

%% file: QPvalue.tex

\section{Individual values for simplicial complexes}

Recall that the player $i$ is \emph{dummy} in the cooperative game $(\Delta, v)$ if $v(S\cup i)=v(S)+v(i)$ for every $i\notin S$ and $S\in \Delta$.
Moreover, we say that $v$ is \emph{monotone} if provided $S \subseteq T \in \Delta$, then $v(S)\leq v(T)$.

Next statement extends Theorem 3.1 of \cite{Shapley-matroids-static} and Theorem 2 and 3 of \cite{Weber-robabilistic-values-for-games} to every simplicial complex. 
%

\noindent
We denote by $0\leq \lambda_i \leq 1$ the rate of participation of the player $i$ in the cooperative game $(\Delta, v)$ and let us rewrite the two main conditions in Theorem 3.1 of \cite{Shapley-matroids-static}: 
\begin{description}
	\item [\hspace{1.55cm} $\lambda_i$-Dummy axiom] If the player $i$ is dummy for $(\Delta, v)$, then $\phi_i(v)=\lambda_i v(i)$;
	\item [\hspace{1cm} Monotonicity axiom] If $v$ is monotone, then $\phi_i(v)\geq 0$.
\end{description}

\begin{theorem}\label{thm:dummy_and_monotone}
	Let $\Delta$ be a simplicial complex on $n$ verticies  and let $\phi_i$ be an individual value for a player $i$.
	The individual value $\phi_i$ is a $\mathbb{R}$-linear function satisfying the $\lambda_i$-Dummy axiom and the Monotonicity axiom if and only if there exists a collection of positive real numbers $\{p_T^i: T\in \Link{i}{\Delta}\}$ with 
		\begin{equation}\label{eq-sum-lambda-simplicial-complex}
			\sum_{T\in \Link{i}{\Delta}} p_T^i = \lambda_i,
		\end{equation}	
		such that for all $v$ in $\charFun$,
		\begin{equation}\label{eq-formula-shapley-simplicial-complex}
			\phi_i(v)=\sum_{T\in \Link{i}{\Delta}} p_T^i (v(T \cup i) - v(T)).		
		\end{equation}
\end{theorem}
\begin{proof} 
	Let us first show the \emph{if part}. 
%
%
	The set $\{\mathbbm{1}_T\}_{T\in \Delta}$ is a basis for the vector space $\charFun$, so every real valued function $v$ can be written uniquely as $v=\sum_{T\in \Delta}\mathbbm{1}_T v(T)$. By the linearity of the function $\phi_i$, we have
	\[
		\phi_i(v)=\phi_i\left(\sum_{T\in \Delta}\mathbbm{1}_T v(T)\right)=\sum_{T\in \Delta}\phi_i(\mathbbm{1}_T)v(T).
	\]
	
	We first start by showing that $\phi_i(\mathbbm{1}_T)=0$ if $T\in\Delta$ but $T\cup i\notin \Delta$.
	Note that $\mathbbm{1}_T=v_T-\hat{v}_T$, that $\phi_i(\hat{v}_T)=\phi_i(\mathbbm{1}_{T\cup i})v(T\cup i)$, and that $\phi_i(v_T)=v_T(i)=0$, because $i$ is dummy for $v_T$.
	Therefore we get $\phi_i(\mathbbm{1}_T)=\phi_i(v_T-\hat{v}_T)=0$.	

	\noindent	
	If $i\in T$, then $\phi_i(\mathbbm{1}_{T})=\phi_i(\mathbbm{1}_T)=\phi_i(\mathbbm{1}_{T\setminus i})$; thus formula \eqref{eq-formula-shapley-simplicial-complex} follows from by setting 
	\[
		p_T^i\define \begin{cases}
				\phi_i(\mathbbm{1}_{T\setminus i}) &  i\in T \\
				-\phi_i(\mathbbm{1}_{T}) &\text{otherwise.}
			\end{cases}.
	\]
	
	Now, let $i$ be a dummy player for $v$, from \eqref{eq-formula-shapley-simplicial-complex}, we get
	\[
		\phi_i(v)=\sum_{T\in \Link{i}{\Delta}} p_T^i v(i) = \left(\sum_{T\in \Link{i}{\Delta}} p_T^i\right) v(i).
	\]
	Equation in \eqref{eq-sum-lambda-simplicial-complex} is obtained by comparing the previous equality with $\phi_i(v)=\lambda_i v(i)$, coming from the $\lambda_i$-dummy property.

	\noindent
	It remains to show that	every $p_T^i$ is a real number greater or equal than zero. For this, consider the monotone function $\hat{v}_T$.
	It is easy to see from \eqref{eq-formula-shapley-simplicial-complex} that  
	\[
		\phi_i(\hat{v}_T)=p_T^i,
	\]
	and the latter is non negative because	$\hat{v}_T$ is a monotone game.
	
	Let us show the \emph{only if part}.
	From \eqref{eq-formula-shapley-simplicial-complex}, we  know that $\phi_i$ is a linear function. Moreover, if $i$ is a dummy player, then $v(S\cup i)=v(S)+v(i)$, so 	
	$$\phi_i(v)=\sum_{T\in \Link{i}{\Delta}} p_T^i v(i) = v(i) \sum_{T\in \Link{i}{\Delta}} p_T^i = \lambda_i v(i).$$
	Finally, if $v$ is monotone, then $v(S \cup i) \geq v(S)$ and, so 	$\phi_i(v)\geq 0$.
\end{proof}

\section{Quasi-probabilistic values}\label{sec:quasi-probabilistic}
 
The necessary and sufficient condition of the previous theorem hints a natural generalization of \emph{probabilistic value} for simplicial complex, given for matroids in \cite{Shapley-matroids-static}.

An individual values $\phi_i$ is a \emph{quasi-probabilistic value} if there exists a collection of positive real numbers 
$$\left\{p_T^i: T\in \Star{i}{\Delta}\right\}$$ 
with 
\begin{equation*}
	\sum_{T\in \Link{i}{\Delta}} p_T^i = \lambda_i,	
\end{equation*}	
such that for all $v$ in $\charFun$,
\begin{equation*}
	\phi_i(v)=\sum_{T\in \Link{i}{\Delta}} p_T^i (v(T \cup i) - v(T)).		
\end{equation*}


The next statement extends Theorem 3.2 in \cite{Shapley-matroids-static} to every simplicial complex, even not pure ones.

\begin{theorem}\label{thm:quasi-probabilistic}
	Let $\Delta$ be a simplicial complex and let $\phi_i$ be an individual value for a player $i$ in $[n]$.
	%

	The individual value $\phi_i$ is quasi-probabilistic if and only if there exists a probability distribution $\{P^i(F_1), \dots, P^i(F_k)\}$ on $\FacetsD$ such that
	\begin{equation}\label{eq-sum-lambda-probability-distribution-simplicial-complex}
		\sum_{F\in \Facet{i}{\Delta}} P^i(F) = \lambda_i		
	\end{equation}	
	and for every $F\in \Facet{i}{\Delta}$ there exists a probabilistic value $\phi_i^F$ defined on the simplex $\bar{F}$ such that for all $v$ in $\charFun$,
	\begin{equation}\label{eq-formula-shapley-weighted-sum-simplicial-complex}
		\phi_i(v)=\sum_{F\in \Facet{i}{\Delta}} P^i(F) \phi_i^F(v|_F),		
	\end{equation}
	where $v|_F$ is the restriction of the cooperative game $(\Delta, v)$ to $(F, v|_F)$.
%
\end{theorem}
\begin{proof}
	Let $i$ be a player and let $P^i$ be a probability distribution on $\Facet{i}{\Delta}$. 	
	Assume that for every $F$ facets in $\Facet{i}{\Delta}$ and for every cooperative game $(F, w)$, the individual value $\phi_i^F$ is a probabilistic value and, then, by Theorem 9 in \cite{Weber-robabilistic-values-for-games}, $\phi_i^F$ is defined as
	\[
		\phi_i^F(w)=\sum_{T\subset F\setminus i} p^i_{F, T} (w(T\cup i)-w(T))
	\]	
	where $p^i_{F, T}$ are non-negative real numbers such that
	\[
		\sum_{T\subset F\setminus i} p^i_{F, T}=1.	
	\] 
	We note that $T\subset F\setminus i$ if and only if $T\in \Link{i}{\Delta}\cap F$.
	
	Then one can show, similarly as in the proof of Theorem 3.2 in \cite{Shapley-matroids-static}, that 
	for all $T$ in $\Link{i}{\Delta}$
	\begin{equation}\label{eq:probability-restriction-simplicial-complex}
		p^i_T=\sum_{F\in \Facet{T\cup i}{\Delta}} P^i(F)p^i_{F, T}.
	\end{equation}
	and also
	\begin{equation}\label{eq:probability-sum-simplicial-complex}
		\sum_{T\in\Link{i}{\Delta}}p^i_T=\sum_{F\in \Facet{T\cup i}{\Delta}} P^i(F).
	\end{equation}
	From the last two equations, one can show easily that the provided conditions \eqref{eq-sum-lambda-probability-distribution-simplicial-complex} and \eqref{eq-formula-shapley-weighted-sum-simplicial-complex} are necessary.

	Let us now focus on the sufficiency of the conditions.
	Consider the restriction $v|_F$ of v to a generic facet $F$ of $\Delta$ and consider an individual probabilistic value $\phi_i^F$ defined on the simplex $\bar{F}$. By definition one has that  
	\begin{equation*}
		\phi_i^F(v|_F)=\sum_{T\subseteq F\setminus i} p_{F, T}^i (v|_F(T \cup i) - v|_F(T)).		
	\end{equation*}
	where $\sum_{T\in \Link{i}{\bar{F}}} p_{F,T}^i=1$.
	Consider also a probability distribution $P^i$ on $\Facet{i}{\Delta}$, that is $P^i(F)\geq 0$ and $\sum_{F\in \Facet{i}{\Delta}} P^i(F)=1$.
	Finally, from the hypothesis, we know that $\phi_i(v)=\sum_{T\subseteq F\setminus i} p_T^i (v(T \cup i) - v(T))$,
	where $\sum_{T\in \Link{i}{\Delta}} p_T^i = \lambda_i$.
	From equation \eqref{eq:probability-sum-simplicial-complex}, one has that if such probability distribution $\{P^i\}$ exists then $\sum_{F\in \Facet{T\cup i}{\Delta}} P^i(F)$ has to be equal $\sum_{T\in\Star{i}{\Delta}}p^i_T$ and, the latter, in the case of a probabilistic value is precisely $\lambda_i$.
	
	It remains to show that $\phi_i(v)=\sum_{F\in \Facet{i}{\Delta}} P^i(F) \phi_i^F(v|_F)$ and for this we need to prove that $P^i$, $p^i_S$ and $p^i_{F, S}$ satisfy equation \eqref{eq:probability-restriction-simplicial-complex}.
	%
	
	\noindent
	Given $S\in \Delta$, let $m_S(\Delta)$ be the number of facets $F$ of $\Delta$ contanining $S$. For all facets F in $\Facet{i}{\Delta}$ denote by
	\[
		P^i(F)\define \sum_{T\in(\Link{i}{\Delta}\cap F)} \frac{p^i_T}{m_{T\cup i}};
	\]
	we also observe that the sum can we simply taken over $\Link{i}{\bar{F}}$ because this set equals $\Link{i}{\Delta}\cap F$. Moreover, for all facets $F$ in $\Facet{T\cup i}{\Delta}$ and for every $T\in \Link{i}{F}$, we define
	\[
		p^i_{F, T}\define \frac{p^i_T}{m_{T\cup i} P^{i}(F)}.
	\]
	Now we substitute the previous value in equation \eqref{eq:probability-sum-simplicial-complex} and one gets 
	\begin{eqnarray*}
		p^i_T&=&\sum_{F\in \Facet{T\cup i}{\Delta}} P^i(F)p^i_{F, T}\\
				&=&\sum_{F\in \Facet{T\cup i}{\Delta}} P^i(F)\frac{p^i_T}{m_{T\cup i} P^{i}(F)}\\
				&=&\sum_{F\in \Facet{T\cup i}{\Delta}} \frac{p^i_T}{m_{T\cup i}}\\
				&=&\frac{p^i_T}{m_{T\cup i}} \sum_{F\in \Facet{T\cup i}{\Delta}} 1.
	\end{eqnarray*}
	The equality holds because by definition 
	$$m_{T\cup i}	= \sum_{F\in \Facet{T\cup i}{\Delta}} 1=|\Facet{T\cup i}{\Delta}|.$$
\end{proof}

As we have done in the introduction, it is important to highlight the discrepancy and the improvement between these results and Theorem 3.2 of \cite{Shapley-matroids-static}:

\begin{romenum}
\item \textbf{i)} Matroids are pure simplicial complexes that satisfies the base exchange properties \cite{Stanley2012b, Stanley1996a,  MR782306, Oxley, Martino2018, Borzi-Martino-D-matroids}. In particular, every facet has the same cardinality, i.e. if the matroid have rank $r$, then every facet has cardinality $r$. In this specific case, facets are called bases.
The integers $m_S(\Delta)$ in the proof of Theorem 3.2 in \cite{Shapley-matroids-static} counts the number of bases (i. e. facets) of the matroid containing $S$. While the concept is the same, in our proof of Theorem \ref{thm:quasi-probabilistic}, $m_S(\Delta)$ takes in consideration facet of different cardinality. 
It is quite remarkable that the probability distribution $\{P(F_i)\}$ in Theorem \ref{thm:quasi-probabilistic} can be defined in the same way regardless of the difference of rank.

\item Not every pure simplicial complex is a matroid. Hence our theorem applies also to the simplicial complex described in Figure \ref{fig:figure-matroid-no}. Simplicial complexes that are not matroids are extremely important in Mathematics; few example can be found in \cite{Martino2015e, MR3503390, Martino-Greco-pinched}.

\item The proof of the statements does not involve any of the matroidal properties \cite{Stanley2012b, Stanley1996a,  MR782306, Oxley}. 
As highlighted in the introduction, none of the matroidal property plays a role in the proof of Theorem \ref{thm:dummy_and_monotone} and \ref{thm:quasi-probabilistic}. This will be different in Section \ref{sec:reducible} where we will need our complex to be pure.
\end{romenum}

\section{Group value and the core}\label{sec:core}

The goal of the individual values is to assign the payoff of the grand coalition $v([n])$ proportionally to the worth of the player. (This is, somehow, the \emph{efficient} scenario: we assume that the values are nor optimistic or pessimistic. We are going to focus on this more in details in the subsequent section.)

\noindent
Therefore, as in the traditional case, we collect all individual values together, in the \emph{group value} $\phi=(\phi_1, \phi_2, \dots, \phi_n)$.
In \cite{Shapley-matroids-static}, they introduce the concept of \emph{basic value}. In view of the results in 
\cite{Martino-Efficiency,Martino-Probabilistic-value}, 
we call this property \emph{reducible}.

\begin{definition}
A \emph{group value} $\phi$ on $\charFun$ is \emph{reducible} if there exists a probability distribution $P$ on the the facets of $\Delta$, $\FacetsD$, such that 
\begin{equation*}
	\phi_i(v)=\sum_{F\in \Facet{i}{\Delta}} P(F) \phi_i^F(v|_F).		
\end{equation*}
where $\phi_i^F$ is a probabilistic value for a cooperative game on the simplex $\bar{F}$ and $v|_F$ is the restriction of the characteristic function $v$ on the simplex $F$, that is $v|_F(S)=v(S)$ for every subset $S$ of $F$.
\end{definition}

\noindent
In other words, a group value is reducible if the group value can be computed reducing the cooperative game to $k=|\FacetsD|$ cooperative games $(\bar{F_1}, v|_{F_1}), \dots, (\bar{F_k}, v|_{F_k})$ on the full simplicies $\bar{F_1}, \dots, \bar{F_k}$.

While every component of a reducible group value $\phi=(\phi_1, \phi_2, \dots, \phi_n)$, is a quasi-probabilistic value, Example 3.1 in \cite{Shapley-matroids-static} shows that the converse in not true. This is because the distribution $P$ on the facets of $\Delta$ needs to be unique for every vertex $i$.

Given a cooperative game $(2^n, v)$, the group value $\phi(v)$ detect a vector in $\mathbb{R}^n$; we could also assume that $\phi_i(v)\geq v(i)$, i.e. each player does not accept any redistribution of the payoff $v([n])$ if it is less than $v(i)$, the amount the player could obtain on its own. We should assume the same for every coalition $T$, that is $\sum _{i\in T}y_{i}\geq v(T)$. Such vectors in $\mathbb{R}^n$ are called \emph{imputations}.

\noindent
Let $x$ and $y$ be two imputations in $\mathbb{R}^n$. We say that $x$ is dominated by $y$ for the cooperative game $([n], v)$, if 
\begin{alphenum}
\item there exists a coalition $T$ such that
	$x_{i}\leq y_{i}$ for all $i\in T$;
\item there exists a player $i$ in $T$ such that $x_{i}<y_{i}$;
\item The coalition $T$ can adopt $y$ as imputation, that is $\sum _{i\in T}y_{i}\geq v(T)$.
\end{alphenum}

\noindent
In other words, the imputation $y$ would be a better deal than $x$ for the players in $T$.
The core is the set of imputations that are not dominated. This notion is used since latest 19-th century, but also more recently appears in several research works; here just a few examples \cite{MR1202057, MR0106116, MR3716594, MR3378384, MR640200, MR2564777, MR3684478}.
Mathematically, this is formulated as it follows.
Assume that for any $(x_1, \dots, x_n)$ in $\mathbb{R}^n$, $x(S)=\sum_{i\in S} x_i$. 
The core of a cooperative game $([n], v)$ is the following set:
\[
	\operatorname{core} ([n], v)\define\{(x_1, \dots, x_n)\in \mathbb{R}^n: x([n])=v([n]),\, x(S)\geq v(S) \forall S \subset [n]\}.
\]

\noindent
Similarly, we define the anticore as
\[
	\operatorname{anticore} ([n], v)\define\{(x_1, \dots, x_n)\in \mathbb{R}^n: x([n])=v([n]),\, x(S)\boldsymbol{\leq} v(S) \forall S \subset [n]\}.
\]

\begin{definition}\label{def:face-polytope}
Given a simplicial complex $\Delta$ over $n$ verticies with facets \FacetsD, the \emph{facet polytope} $Q_{\Delta}$ is the convex hull in $\mathbb{R}^n$ of vectors $e_F=\sum_{i\in F}e_i$ for every facet F in \FacetsD, where $e_i$ is a standard orthonormal basis of $\mathbb{R}^n$:
\[
	Q_{\Delta}=\operatorname{convex} \{e_F: F\in \FacetsD\}.
\]
\end{definition}

The facet polytope and the core do not just share the ambient space. Indeed, when $\Delta=M$ is a matroid, then $Q_{M}$ is the base polytope, also called matroid polytope, (see for instance Chapter 9 of \cite{Zie}) and Edmonds \cite{Edmonds-submodular-functions} has shown that $Q_{M}$ is the anti-core of of the game $([n], r)$, where $r$ is the rank function of the matroid $M$,
\begin{equation}\label{eq:anti-cor-poly}
	\operatorname{anticore} ([n], \operatorname{rk}_{M}) = Q_{M}.
\end{equation}

\noindent
This was also reproved in Theorem 2.3 of \cite{Shapley-matroids-static}.

Moreover, when $\Delta=2^{[n]}$ is the full simplex, Weber \cite{Weber-robabilistic-values-for-games} shows the core of the game is contained in the so called Weber set. Let us recall the definition of Weber set and this result.
Let $S_n$ be the set of bijective function from $[n]$ to $[n]$. 
For every such bijective function $\pi$, we list the image of each element in a unique ordered set $\pi=(\pi(1), \dots, \pi(n))$ and we denote for any $i\in [n]$ by $\pi^i\define \{j\in [n]: \pi(j) < \pi(i)\}$. 
Let $(2^n, v)$ be a cooperative game and we define the marginal worth vector $a^{\pi}(v)$ as the imputation satisfying $a^{\pi}_i(v)=v(\pi^i\cup i)-v(\pi^i)$ for all $i$ in $[n]$. 
Let the Weber set be the set of all imputations which are associated with $v$ by some random-order value (that is equivalent, under the efficiency axiom to an efficient probabilistic group value), that is:
\[
	\operatorname{Weber} (2^n, v) \define \operatorname{convex}\{a^{\pi}(v): \pi\in S_n \}
\]

\begin{theorem}[Thm 14 in \cite{Weber-robabilistic-values-for-games}, Shapley \cite{Shapley-core-convex} and Ichiishi \cite{MR640200}]
	For any cooperative game $([n], v)$, $\operatorname{core} ([n], v) \subseteq \operatorname{Weber} ([n], v)$.
	The equality holds if the game is convex.
\end{theorem}

We denote by $\operatorname{Prob}{\Delta}$ the set of probability distribution over the set of facets of $\Delta$, that is:
\[
	\operatorname{Prob}{\Delta}\define\{P\in \mathbb{R}^{|\FacetsD|}: P(F)\geq 0 \mbox{ and } \sum_{F\in \FacetsD}P(F)=1\}.
\]
Associated to a probability distribution $P$ in $\operatorname{Prob}{\Delta}$, following Section 4 of\cite{Shapley-matroids-static}, we also define the \emph{probabilistic participation influence} $w^P(T)$ of $T$ as 
\[
	w^P(T)=\sum_{T\subset F\in \FacetsD} P(F).
\]
If $T$ is not in $\Delta$ then $w^P(T)=0$; this can also be seen using the previous definition as $T$ is not a subset of any facet of $\Delta$.

\noindent
In Proposition 4.1 of \cite{Shapley-matroids-static}, using Edmonds \cite{Edmonds-submodular-functions} results in \eqref{eq:anti-cor-poly}, they are able to prove that 
\[
	\operatorname{anticore} ([n], \operatorname{rk}_{M}) \overset{\operatorname{\scriptscriptstyle \eqref{eq:anti-cor-poly}}}=Q_{M} =\{w^P: P\in \operatorname{Prob}{M}\}.
\]
While the same equality between the anticore and the facet polytope in \eqref{eq:anti-cor-poly} does not hold also for a generic simplicial complex, we are able to prove that $Q_{\Delta} =\{w^P: P\in \operatorname{Prob}{\Delta}\}$.
Aside of the generalization per se of the result in \cite{Shapley-matroids-static}, one perk of next proposition is that we do not use Edmond result or any connection with the anticore of the cooperative game. 

\begin{theorem}
	Let $\Delta$ be a simplicial complex and let $r$ be its rank function. Then
	\[
		Q_{\Delta}=\{w^P: P\in \operatorname{Prob}{\Delta}\}.
	\]
\end{theorem}
\begin{proof}
	Let $q$ be an element in $Q_{\Delta}$: $q$ can be written as a convex linear combination of the incidence vectors of the facets $\{e_F\}$, that is $q=\sum_F \alpha_F e_F$ with $\alpha_F\geq 0$ and $\sum_F\alpha_F=1$. If we set $P(F)\define \alpha_F$, then $q=w^P$. 
	The opposite inclusion follows similarly.
\end{proof}

\section{Reducible quasi-probabilistic values}\label{sec:reducible}
In the traditional theory, a cooperative game is  defined on the full simplex $\Delta=2^{[n]}$ and the Shapley values are probabilistic values arising from the following common point of view among the players: The player $i$ joins a coalitions of different sizes with the same probability; All coalition of the same size are equally likely.

Thus, the every player has $n$ possibilities to choose the size of a coalition (the joint coalitions may have cardinality $k= 0\leq k\leq n-1$) and, further, there are ${n-1 \choose k}$ choices among all sets (coalitions) of cardinality $k$ among the other players $[n]\setminus i$.
Therefore, one defines the Shapley values for the player $i$ as 
\[
	\operatorname{Shapley}_i(v)=\sum_{T\subseteq [n]\setminus i} \frac{1}{n}\frac{|T|!(n-|T|-1)!}{(n-1)!} (v(T\cup i)-v(T)).
\]

The classical Shapley Theorem characterizes Shapley values as follow:
\begin{theorem}[Shapley's Theorem, see for instance \cite{Weber-robabilistic-values-for-games}]	
	Let $(n, v)$ belong to a cone $
\mathcal{I}$ of cooperative games containing the carrier games $\mathcal{C}$ and $\hat{\mathcal{C}}$.
	Assume that if $v\in\mathcal{I}$, then the permuted game $\pi \cdot v$ is also in $\mathcal{I}$ for every permutation $\pi$ of $[n]$.
	Let $\phi$ be a group value.

	%

	 If each $\phi_i$ is a linear function that satisfies the dummy axiom and the monotonicity axiom and if the symmetric and the efficiency axioms hold for the groups value  $\phi$, then for every cooperative game $(n, v)$ in the domain of $\phi$ and every $i$ in $[n]$, 
	\[
		\phi_i(v)=\operatorname{Shapley}_i(v).
	\]
\end{theorem}

In \cite{Shapley-matroids-static}, they characterize when in individual value can be written as weighted sum of Shapley values defined on the bases of a matroids.

All these results together allow us to generalize the main theorem of \cite{Shapley-matroids-static} to pure simplicial complex. 

\begin{theorem}
	Let $\Delta$ be a \textbf{pure} simplicial complex on $n$ verticies and and let $\phi_i$ be the individual value for a player $i$.
	The group value $\phi$ is reducible and it decomposes as the weighted sum of Shapley values,
	\begin{equation*}
		\phi_i(v)=\sum_{F\in \Facet{i}{\Delta}} P(F)\operatorname{Shapley}^F(v|_F)	
	\end{equation*}
	if and only if each $\phi_i$ is a linear function that satisfies the $w^P(i)$-dummy axiom and the group value fulfills the following two axioms:
	 \begin{description}
		\item [\hspace{0.7cm} Substitution for carrier games] For every coalition $T$ and for every pair of players, one has $\phi_i(v_T)=\phi_j(v_T)$;
		\item [\hspace{0.7cm} Probabilistic efficiency] For every cooperative game $(\Delta, v)$, $\sum_i \phi_i(v)=\sum_{F\in \FacetsD}P(F)v(F)$.
\end{description}
\end{theorem} 
\begin{proof}
	Part of the characterization, follows form the ones in Theorem \ref{thm:dummy_and_monotone} and Theorem \ref{thm:quasi-probabilistic}. 
	
	The \emph{Substitution for carrier games} axioms imposes that locally each individual value has to be of a Shapley one with maximal facet of the same cardinality. Finally, the \emph{Probabilistic efficiency} provides the requested weighted sum.
\end{proof}